\documentclass[12pt]{amsart}

\usepackage{amssymb}
\usepackage{hyperref}
\usepackage{graphicx}
\usepackage{enumerate}
\usepackage{vmargin}
\setpapersize{USletter}
\setmargrb{1in}{1in}{1in}{1in} 

\newtheorem*{thm*}{Theorem}
\newtheorem{thm}{Theorem}
\newtheorem{cor}[thm]{Corollary}

\newcommand{\Z}{\mathbb{Z}}

\newcommand{\C}{\mathbb{C}}

\DeclareMathOperator{\Hom}{Hom}

\newcommand{\defining}[1]{\textbf{#1}}

\title{Topological criteria for schlichtness}
\author{Zach Teitler}
\email{zteitler@boisestate.edu}
\address{Department of Mathematics \\ 1910 University Drive \\ Boise State University \\ Boise, ID 83725-1555 \\ USA}
\date{December 2, 2010}

\begin{document}

\bibliographystyle{amsalpha}       

\begin{abstract}
We give two sufficient criteria for schlichtness of envelopes of holomorphy in terms of topology.
These are weakened converses of results of Kerner and Royden.
Our first criterion generalizes a result of Hammond in dimension 2.
Along the way we also prove a generalization of Royden's theorem.
\end{abstract}

\maketitle

Let $\Omega \subseteq \C^n$ be a domain.
The \defining{envelope of holomorphy} of $\Omega$ is a pair $(\tilde\Omega, \pi)$
consisting of a connected Stein manifold $\tilde\Omega$ 
and a locally biholomorphic map $\pi \colon \tilde\Omega \to \C^n$,
together with a holomorphic inclusion $\alpha \colon \Omega \to \tilde\Omega$,
characterized by the following properties:
$\pi \circ \alpha$ is the identity, and
each holomorphic function $f$ on $\Omega$ has a unique holomorphic extension $F_f$ on $\tilde\Omega$ with $f = F_f \circ \alpha$.
Let $\Omega' = \pi( \tilde \Omega )$
and let $i = \pi \circ \alpha \colon \Omega \to \Omega'$.
The envelope of holomorphy $(\tilde \Omega, \pi)$ is \defining{schlicht} if $\pi \colon \tilde\Omega \to \Omega'$ is biholomorphic.
One would like to give conditions on $\Omega$ to have a schlicht envelope of holomorphy.

Two results of Kerner and Royden lead to necessary conditions.
Kerner~\cite{Kerner} has shown that
$\alpha_* \colon \pi_1(\Omega) \to \pi_1(\tilde\Omega)$ is surjective.
Royden~\cite{MR0152675} has shown that
$\alpha^* \colon H^1(\tilde\Omega ; \Z) \to H^1(\Omega ; \Z)$ is injective.
It follows trivially that if $(\tilde \Omega, \pi)$ is schlicht, so $\tilde\Omega = \Omega'$,
then $i_* \colon \pi_1(\Omega) \to \pi_1(\Omega')$ is surjective
and $i^* \colon H^1(\Omega' ; \Z) \to H^1(\Omega ; \Z)$ is injective.

Neither of these conditions is sufficient, by a result of Fornaess and Zame~\cite{FZ} (see \cite[\textsection3]{hammond}).
Following an idea of Hammond \cite{hammond} one may seek sufficient conditions by
adjoining to surjectivity of $i_*$ (or injectivity of $i^*$) the assumption that $\pi \colon \tilde\Omega \to \Omega'$ is a covering space.
This strong assumption is still reasonable, as covering maps certainly occur among envelopes of holomorphy ---
indeed, Fornaess and Zame show in \cite{FZ}
that for any covering map $\pi \colon \tilde\Omega \to \Omega'$ there is a domain $\Omega \subseteq \Omega'$
with envelope of holomorphy $(\tilde\Omega, \pi)$.

Specifically, Hammond has shown that, in dimension $n=2$,
if $i_* \colon \pi_1(\Omega) \to \pi_1(\Omega')$ is surjective and $\pi \colon \tilde\Omega \to \Omega'$ is a covering map,
then $(\tilde\Omega, \pi)$ is schlicht.
We give an elementary proof of Hammond's theorem in all dimensions $n \geq 2$.
In addition we give a sufficient condition for schlichtness in terms of injectivity of $i^*$ on cohomology,
again assuming $\pi$ is a covering map.
Along the way we give an alternative proof of Royden's theorem which also extends it to other coefficient groups than $\Z$.

\begin{thm}
If $\pi$ is a covering map and $i_* \colon \pi_1(\Omega) \to \pi_1(\Omega')$ is surjective
then $(\tilde\Omega, \pi)$ is schlicht.
\end{thm}
This extends the theorem of Hammond for dimension $n=2$.
Hammond's proof relies on a result of Jupiter~\cite{MR2221089} which is special to dimension $2$.

\begin{proof}
The number of sheets of the covering map $\pi$ is equal to the index of $\pi_*(\pi_1(\tilde\Omega))$ in $\pi_1(\Omega')$
(see, for example, \cite[Prop.~1.32]{MR1867354}).
The surjectivity of $i_* = \pi_* \circ \alpha_*$ implies $\pi_*$ is surjective.
Hence the index of the image subgroup is $1$, so $\pi \colon \tilde\Omega \to \Omega'$ is $1$-sheeted, i.e., a homeomorphism.
Since $\pi$ is a holomorphic homeomorphism it is biholomorphic and so $\tilde\Omega$ is schlicht.
\end{proof}
Compare the more technical proof in \cite{hammond}.

The cohomology in Royden's result is \v{C}ech cohomology with coefficients in the sheaf of locally constant $\Z$-valued functions.
Since our spaces are manifolds, \v{C}ech cohomology coincides with singular cohomology (with coefficients in $\Z$); see, for example, \cite[Thm.~73.2]{MR755006}.
Recall also that by the universal coefficient theorem,
$H^1(X ; G) = \Hom(\pi_1(X), G)$, for a path-connected space $X$ and abelian coefficient group $G$
\cite[pg.~98]{MR1867354}.

Before we go on, observe that this proves Royden's theorem as a consequence of Kerner's theorem
and extends it to other coefficient groups.
\begin{thm}[Royden]
For any abelian group $G$, $\alpha^* \colon H^1(\Omega ; G) \to H^1(\tilde\Omega ; G)$ is injective.
\end{thm}
\begin{proof}
Since $\alpha_* \colon \pi_1(\Omega) \to \pi_1(\tilde\Omega)$ is surjective,
$\alpha^* \colon \Hom(\pi_1(\Omega), G) \to \Hom(\pi_1(\tilde\Omega), G)$ is injective
and these $\Hom$ groups coincide with $H^1(\Omega ; G)$, $H^1(\tilde\Omega ; G)$.
\end{proof}
Royden proves this for $G = \Z$ using \v{C}ech cohomology, in particular the exponential short exact sequence (hence the restriction to $G = \Z$).

We get the following.

\begin{thm}
If $\pi$ is a covering map,
$\pi_1(\Omega')$ is nilpotent,
and
$i^* \colon H^1(\Omega' ; G) \to H^1(\Omega ; G)$ is injective for every abelian group $G$,
then $(\tilde\Omega, \pi)$ is schlicht.
\end{thm}

\begin{proof}
Since $i^* = \alpha^* \circ \pi^*$ is injective, $\pi^*$ is injective as well.
Via $\pi_*$ we regard $\pi_1(\tilde\Omega)$ as a subgroup of $\pi_1(\Omega')$.
Recall that if $H$ is any nilpotent group then every maximal proper subgroup $N$ of $H$ is normal and has prime index (see \cite[Thm.~5.40]{MR1307623}),
and in particular $H/N$ is abelian.
If $\pi_1(\tilde\Omega) \subsetneqq \pi_1(\Omega')$ there is a maximal subgroup $\pi_1(\tilde\Omega) \subseteq N \subsetneqq \pi_1(\Omega')$
and hence a surjection $\pi_1(\Omega') \to G = \pi_1(\Omega') / N$ to an abelian group with $\pi_1(\tilde\Omega)$ mapping to zero.
This surjection is nonzero and lies in the kernel of
\[
  \pi^* \colon H^1(\Omega' ; G) = \Hom(\pi_1(\Omega'), G) \to \Hom(\pi_1(\tilde\Omega), G) = H^1(\tilde\Omega ; G)
\]
for the abelian group $G = \pi_1(\Omega') / N$, contradicting the injectivity of $\pi^*$.

It follows that $\pi_1(\tilde\Omega) = \pi_1(\Omega')$.
As before this implies $\pi$ is a degree $1$ covering map, hence a biholomorphism.
\end{proof}

It is not necessary to assume $i^*$ is injective when coefficients are taken in any abelian group $G$.
It would be enough to assume $i^*$ is injective when coefficients are taken in any finite cyclic group,
in any abelian quotient $G$ of $\pi_1(\Omega')$,
or even just in a single abelian quotient $G = \pi_1(\Omega') / N$ for some proper normal subgroup $N$
containing $\pi_1(\tilde\Omega)$.

If in addition $\pi \colon \tilde\Omega \to \Omega'$ is a normal covering space then $\pi_1(\tilde\Omega) \subseteq \pi_1(\Omega')$
is a normal subgroup and we can take $G$ to be an abelian quotient of
$\pi_1(\Omega') / \pi_1(\tilde\Omega)$, which is the group of deck transformations.

\begin{cor}
Suppose $\pi$ is a normal covering map with deck transformation group $H$.
If there exists a nonzero abelian quotient $G$ of $H$ such that
$i^* \colon H^1(\Omega' ; G) \to H^1(\Omega ; G)$ is injective,
then $(\tilde\Omega, \pi)$ is schlicht.
\end{cor}


I thank Chris Hammond for explaining his theorem to me, and Emil Straube, Craig Westerland, and Jens Harlander
for helpful and patient conversations.

\end{document}